\documentclass[letterpaper]{article}

\usepackage[french,english]{babel}
\usepackage[utf8x]{inputenc}
\usepackage[T1]{fontenc}
\usepackage[affil-it]{authblk}

\usepackage[letterpaper,top=3cm,bottom=3cm,left=3cm,right=3cm,marginparwidth=1.75cm]{geometry}

\usepackage{amsmath}
\usepackage{amsthm}
\usepackage{amssymb}
\usepackage[all,cmtip]{xy}
\usepackage{graphicx}
\usepackage[colorinlistoftodos]{todonotes}
\usepackage[colorlinks=true, allcolors=blue]{hyperref}
\usepackage{tikz-cd}
\usepackage{theoremref}

\theoremstyle{plain}
\newtheorem{theorem}{Theorem}[section]
\newtheorem{lemma}[theorem]{Lemma}
\newtheorem{prop}[theorem]{Proposition}
\newtheorem{cor}[theorem]{Corollary}

\theoremstyle{definition}
\newtheorem{exx}[theorem]{Example}
\newtheorem{deff}[theorem]{Definition}
\newtheorem{rmk}[theorem]{Remark}

\newcommand{\F}{\mathbb{F}}
\newcommand{\Q}{\mathbb{Q}}
\newcommand{\C}{\mathbb{C}}

\newcommand{\CR}{\mathbb{C}R_{\mathbb{C}}}

\newcommand{\E}{\mathbb{E}}
\newcommand{\Z}{\mathbb{Z}}

\newcommand{\R}{\mathbb{R}}
\newcommand{\RF}{kR_{\mathbb{F}}}
\newcommand{\RQ}{kR_{\mathbb{Q}}}
\newcommand{\RQG}{k R_{\mathbb{Q},G}}

\newcommand{\RFG}{kR_{\mathbb{F},G}}
\newcommand{\ER}{\widehat{kR_{\mathbb{Q}}}}
\newcommand{\ERG}{\widehat{kR_{\mathbb{Q},G}}}

\newcommand{\ESZERO}{\widehat{\mathcal{Z}}}
\newcommand{\ESSUP}{\widehat{\mathcal{S}}upp}

\title{Essential support of Green biset functors via morphisms.}
\author{Benjam\'in Garc\'ia}
\affil{Department of Mathematics\\
University of California\\ 
Santa Cruz, CA 95064\\ 
bgarciah@ucsc.edu}

\begin{document}
\maketitle

\begin{abstract}
We present a very natural but yet useful criterion to detect vanishing of essential algebras of a Green biset functor $A$ by means of morphisms. We introduce the morphisms $Inf:A \rightarrow A_G$ and $Res:A_G\rightarrow A$ to prove that the class of groups for which the essential algebras do not vanish is the same for $A$ and any shifted functor $A_G$. We use the kernel of $Res$ to give a characterization of the seeds of simple $A_G$-modules which are not obtained as pullback of simple $A$-modules.

\end{abstract}

\begin{otherlanguage}{french}
\begin{abstract}
Nous présentons un critère très naturel mais pourtant utile pour détecter la nullité des algèbres essentielles d'un foncteur à bi-ensembles de Green $A$ via morphismes. Nous introduisons les morphismes $Inf:A\rightarrow A_G$ et $Res:A_G\rightarrow A$ pour démonstrer que la classe des groupes pour lesquelles les algèbres essentielles ne sont pas nulles est la même pour $A$ et tous les foncteurs $A_G$. Nous utilisons le noyau de $Res$ pour donner une caractérisation des germes de ces $A_G$-modules simples qui ne sont pas obtenus comme retrait de $A$-modules simples.

\textit{Mots clés}: Foncteur à bi-ensembles de Green, algèbre essentielle, construction de Yoneda-Dress.
\end{abstract}
\end{otherlanguage}

\section{Introduction}

A natural problem to consider when studying a category of linear functors is to give a classification of its simple objects by means of parameters which help to better understand the effect of these functors on the objects and morphisms of the domain category. For simple biset functors over an admissible subcategory $\mathcal{D}$, Bouc found a solution by giving a bijection between the isomorphism classes of simple biset functors and the equivalence classes of pairs $(H,V)$ called \textit{seeds}, where $H$ is a group in $Ob(\mathcal{D})$ and $V$ is a simple $kOut(H)$-module \cite[Ch. 4]{SB1}.

The ideas for the classification of simple biset functors can be extended to the case of modules over a general Green biset functor $A$ by considering them as linear functors over the associated category $\mathcal{P}_A$. Now, a \textit{seed} in $\mathcal{P}_A$ is a pair $(H,V)$ consisting of a group $H$ for which the \textit{essential algebra} $\widehat{A}(H)$ is non-zero and a simple $\widehat{A}(H)$-module $V$. An isomorphism relation on the class of seeds is then defined, in a way that if $(H,V)$ and $(K,W)$ are isomorphic, then $H$ and $K$ are isomorphic. Then for any seed $(H,V)$ there is a simple module $S_{H,V}$ for which $H$ is a minimal group and $S_{H,V}(H)\cong V$ as $\widehat{A}(H)$-modules, isomorphic seeds give rise to isomorphic simple modules and any simple module is isomorphic to one of the form $S_{H,V}$. Bouc conjectured that the assignment $(H,V)\mapsto S_{H,V}$ would induce a one-to-one correspondence between the set of isomorphism classes of seeds and the set of isomorphism classes of simple $A$-modules, what was proved wrong later by Romero in \cite{ROM2}, where she provides an example of a simple $kB^{C_4}$-module having two non-isomorphic minimal groups. 

Romero proved in \cite[Proposition 4.2]{ROM1} that if $A$ satisfies that any two minimal groups for a simple $A$-module are isomorphic, then the correspondence  $[(H,V)]\mapsto [S_{H,V}]$ is actually a bijection between the sets of isomorphism classes. Some well-known Green biset functors satisfy this uniqueness condition, e.g., the Burnside functor $kB$ \cite{SB1}, the functor of linear representations $\RF$ for fields $k$ and $\F$ of characteristic zero and its shifted functors (\cite{LB}, \cite{paper}), and the fibered Burnside functor $kB^{C_p}$ for a prime number $p$ \cite{ROM2}. In this case, a better understanding of the correspondence between seeds and simple modules requires to go further in the study of the essential algebras and their simple modules.

This note is devoted to the problem of finding the class of groups $H$ for which $\widehat{A}(H)$ is non-zero, or as we call it, the \textit{essential support} of $A$, denoted by $\ESSUP(A)$. In Section 3, we observe that a morphism of Green biset functors $f:A\longrightarrow C$ induces a morphism of $k$-algebras $\widehat{f}:\widehat{A}(H)\longrightarrow \widehat{C}(H)$ for any $H$, proving that $\ESSUP(C)$ is contained in $\ESSUP(A)$. This criterion had already been used in \cite{paper} to prove that $\ESSUP{\RF}$ consists only of cyclic groups by means of the extension morphism. A further consequence is presented in Section 4, where we prove that $\ESSUP(A_G)=\ESSUP(A)$ for any $G$ by means of the morphisms $Inf:A\longrightarrow A_G$ and $Res:A_G\longrightarrow A$. Finally, since the restriction morphism $Res$ is surjective in any component, we use its kernel $\kappa_G$ to give a characterization of the seeds of those simple $A_G$-modules which cannot be obtained as the pullback of simple $A$-modules via $Res$.

\section{Generalities}

We provide some generalities on the theory of biset functors, for details we refer to Chapters 1 to 4 of \textit{Biset Functors for Finite Groups} by Bouc \cite{SB1}. 

Throughout this note, the letters $G$, $H$, $K$ and $L$ denote finite groups. The symbols $\Z$, $\Q$, $\R$ and $\C$ stand for the rings of integer, rational, real and complex numbers respectively. In general, if $\F$ is a field, $\overline{\F}$ stands for its algebraic closure inside of a fixed algebraically closed extension $\Omega$. All the rings and algebras are assumed to be associative and unital, and ring and algebra homomorphisms send units to units, property which is very important in the development of this text. 

Given a group $G$, a left $G$-set $X$ is a finite set on which $G$ acts on the left and a morphism of left $G$-sets is just an $G$-equivariant function between $G$-sets, while $_Gset$ stands for the category of left $G$-sets and equivariant functions. If $X$ and $Y$ are $G$-sets, their disjoint union $X\sqcup Y$ is again a $G$-set for the action of $G$ on each component, and the cartesian product $X\times Y$ is also a $G$-set for the diagonal action of $G$. The disjoint union and the product behave like sum and multiplication up to isomorphism, inducing a semiring structure on the set of isomorphism classes of $G$-sets which extends to a commutative ring structure in the Grothendieck group of $_Gset$, known as the \textit{Burnside ring} of $G$ and denoted by $B(G)$. In a similar manner one can define right $G$-sets and their category $set_G$.

If $H$ is another group, an $(H,G)$-biset is a finite set which is both a left $H$-set and a right $G$-set in a way that the actions commute, and again, a morphism of bisets is a function which respects both actions. We often write \textit{"$_HX_G$ is a biset"} to say that $X$ is an $(H,G)$-biset. If $_HX_G$ is a biset, its opposite biset is the $(G,H)$-biset $X^{op}$ whose underlying set is the same as for $X$ but with actions given by $g\cdot x\cdot h:=h^{-1}xg^{-1}$ for $x\in X^{op}$, $g\in G$ and $h\in H$. Bisets can also be added by disjoint union, and $B(H,G)$ stands for the Grothendieck group of the category $_Hset_G$ of $(H,G)$-bisets and equivariant functions. Note that $B(H,G)$ and $B(H\times G)$ are naturally isomorphic as groups.

Some important examples of bisets arise as follows: if $\psi:H\longrightarrow G$ is a group homomorphism, then $G$ becomes an $(H,G)$-biset $_{H^{\psi}}G_G$ by the rule $h\cdot g\cdot g_0=\psi(h)gg_0$ for $g\in {_{H^{\psi}}G_G}$, $h\in H$ and $g_0\in G$, while $_GG_{^{\psi}H}$ stands for its opposite. Particular cases of this are the \textit{basic bisets}: if $H\leq G$, then $Res^G_H={_{H^{\leq}}G_G}$ and $Ind^G_H={_GG_{^{\geq}H}}$ are known as the \textit{restriction from $G$ to $H$} and the \textit{induction from $H$ to $G$} respectively; if $N\unlhd G$ and $\pi:G\longrightarrow G/N$ is the canonical projection, then $Inf^G_{G/N}={_{G^{\pi}}G/N_{G/N}}$ and $Def^G_{G/N}={_{G/N}G/N_{^{\pi}G}}$ are known as the \textit{inflation from $G/N$ to $G$} and the \textit{deflation from $G$ to $G/N$ }respectively; when $\psi:H \longrightarrow G$ is an isomorphism, $_GG_{^{\psi}H}$ is often denoted by $Iso(\psi)$.

If $_KY_H$ and $_HX_G$ are bisets, there is a left action of $H$ on $Y\times X$ given by $h\cdot (y,x)=(y\cdot h^{-1},h\cdot x)$, and the \textit{composition of $Y$ with $X$}, that we denote by $Y\times_H X$ or $Y\circ X$, is the set of orbits of this action seen as a $(K,G)$-biset for $k\cdot [y,x]\cdot g=[k\cdot y,x\cdot g]$, where $k\in K$, $g\in G$ and $[y,x]$ stands for the $H$-orbit of an element $(y,x)$ of $Y\times X$. Composition of bisets commutes with disjoint union and it is associative up to isomorphism, and also $H\circ X\cong X\cong X\circ G$ for any biset $_HX_G$. Passing to Burnside groups, this operation induces a family of biadditive applications $B(K,H)\times B(H,G)\longrightarrow B(K,G)$, denoted by $(\beta,\alpha)\mapsto \beta\circ \alpha$, for all groups $K$, $H$ and $G$, which naturally leads to the definition of the \textit{biset category} $\mathcal{C}$, whose objects are all finite groups with hom-sets given by $\mathcal{C}(G,H)= B(H,G)$ and identity element $Id_G=[G]$ in $B(G,G)$ for the composition just defined. It is immediate that the biset category is a preadditive category, and as such it is generated by the isomorphism classes of the basic bisets.

\begin{deff}
Let $\mathcal{D}$ be a preadditive subcategory of $\mathcal{C}$ and $k$ be a commutative ring, and let $k\mathcal{D}$ denote the $k$-linearization of $\mathcal{D}$. A \textit{biset functor for $\mathcal{D}$ over $k$} is a $k$-linear from $k\mathcal{D}$ to $k-Mod$. We write $\mathcal{F}_{\mathcal{D},k}$ for the category of biset functors for $\mathcal{D}$ over $k$ and morphisms given by natural transformations.
\end{deff}

Biset functors provide a general framework for the study of the functorial properties of many reprentation groups and rings arising in representation theory of finite groups. Some well-known  examples are the Burnside functor $kB$, the functor of linear representations $\RF$, the monomial Burnside functor $kD$, the functor $B^{\times}$ of units of the Burnside ring and global Mackey functors.

\section{Green biset functors and essential support}

A Green biset functor is a biset functor with an additional multiplicative structure. Although there are many equivalent definitions, the one that we provide is more convenient for the treatment of the simple objects in its category of modules. For more details, we refer to Bouc \cite[Chapter 8]{SB1} or Romero \cite{ROM1}.

Let $\mathcal{D}$ be a full subcategory of $\mathcal{C}$ with the property that any group isomorphic either to a subquotient of an object of $\mathcal{D}$ or to a product of two objects of $\mathcal{D}$ is an object of $\mathcal{D}$ too. Examples of such subcategories are the whole biset category and the full subcategories having by objects all nilpotent groups, $p$-groups and $p'$-groups for a prime number $p$, respectively. 

\begin{deff} \thlabel{GBF}
For any groups $L$, $K$ and $H$, let $\alpha_{L,K,H}:L\times (K\times H)\longrightarrow (L \times K)\times H$, $\lambda_H:1\times H\longrightarrow H$ and $\rho_H: H\times 1\longrightarrow H$ be the canonical isomorphisms. A \textit{Green biset functor} is a biset functor $A$ together with \textit{bilinear products} 
$$\xymatrix{{A(K)\times A(H)}\ar[r]^-{\times} &{A(K \times H)}},$$ 
denoted by $(a,b)\mapsto a\times b$, for any pair of objects $K$ and $H$ of $k\mathcal{D}$, and an element $\epsilon_A \in A(1)$, satisfying the following conditions:
\begin{enumerate}
\item (Associativity) Let $L$, $K$ and $H$ be objects of $k\mathcal{D}$. Then for all $a\in A(L)$, $b\in A(K)$ and $c\in A(H)$,
$$(a\times b)\times c=A(Iso(\alpha_{L,K,H}))(a\times (b\times c)).$$
\item (Identity element) Let $H$ be an object of $k\mathcal{D}$. Then for any $a \in A(H)$,
$$A(Iso(\lambda_H))(\epsilon_A \times a)=a=A(Iso(\rho_H))(a \times \epsilon_A).$$
\item (Functoriality) If $\alpha: K\rightarrow T$ and $\beta:H\rightarrow L$ are morphisms in $k\mathcal{D}$, then for all $a\in A(K)$ and $b\in A(H)$,
$$A(\alpha \times \beta)(a\times b)=A(\alpha)(a) \times A(\beta )(b).$$
\end{enumerate}

If $C$ is another Green biset functor, a \textit{morphism of Green biset functors} is a morphism of biset functors $f:A\longrightarrow C$ such that $f_{K\times H}(a \times b)= f_K(a) \times f_H(b)$ for all $a \in A(K)$ and $b \in A(H)$, and $f_1(\epsilon_A)=\epsilon_C$.
\end{deff}

An equivalent definition states that a biset functor $A$ is a Green biset functor if $A(H)$ is a $k$-algebra for any $H$, plus some  conditions on the basic morphisms. Such structure can be obtained from \thref{GBF} by setting a multiplication
$$ab:=A(Iso(\delta_H^{-1})\circ  Res^{H\times H}_{\Delta(H)})(a\times b)$$
in $A(H)$, for all $a,b\in A(H)$, where $\delta_H:H\longrightarrow \Delta(H)$ is the natural isomorphism, for which the element $A(Inf_1^H)(\epsilon)$ acts as the identity. See Bouc \& Romero \cite[Subsection 1.1]{SerNad} for a clear exposition on this equivalence.

\begin{exx}
The Burnside functor is a Green biset functor for the maps $kB(K)\times kB(H)\longrightarrow kB(K\times H)$ induced by $([Y],[X])\mapsto [Y\times X]$ for any $Y\in\; _Kset$ and $X\in\; _Hset$. The functor $kB$ is an initial object in the category of Green biset functors.
\end{exx}

\begin{exx}
If $char \F=0$, the functor of $\F$-linear representations $\RF$ is a Green biset functor for the bilinear maps induced by external products of modules. The $H$-th component of the only morphism of Green biset functors $\lambda: kB\longrightarrow \RF$ sends the isomorphism class of an $H$-set $X$ to the character of the permutation module $\F X$ and it is known as the \textit{linearization morphism}. If $\E/ \F$ is a field extension, the extension of scalars from $\F$ to $\E$ on $\F H$-modules induces a map ${^{\E}\eta_H}={^{\E/\F}\eta_H}:\RF(H)\longrightarrow kR_{\E}(H)$ for any $H$, these maps define a morphism of Green biset functors ${^{\E}\eta}={^{\E/\F}\eta}: \RF \longrightarrow kR_{\E}$, known as the \textit{$\E$-extension morphism} (see G. \cite[Section 3]{paper}).
\end{exx}

\begin{exx}
A \textit{left ideal} $I$ of a Green biset functor $A$ is a subfunctor such that $a \times x \in I(K\times H)$ for any $a\in A(K)$ and $x\in I(H)$, and \textit{right} and \textit{two-sided} ideals are defined in a similar way. The kernel of a morphism of Green biset functors is always a two-sided ideal of the domain. It is straightforward to see that $I$ is a left (resp. right, two-sided) ideal if and only if it is a subfunctor of $A$ such that $I(H)$ is a left (resp. right, two-sided) ideal of $A(H)$ for any $H$.
\end{exx}

Other examples of Green biset functors are the \textit{functor of monomial representations} $kD$, or more general, the \textit{$A$-fibered Burnside functor} $kB^A$ for an abelian group $A$ (since $kD\cong kB^{\C^{\times}}$) \cite{BolCos}, and the functor $kpp_{\F}$ of rings of $p$-permutation modules for algebraically closed fields $k$ and $\F$ with $char k=0$ and $char \F=p>0$ \cite{MD}.

\begin{deff}
If $H$ is a finite group, we write $\overrightarrow{H}$ for the $(H\times H,1)$-biset $H$ with actions $(h_1,h_2)\cdot h \cdot 1=h_1hh^{-1}_2$ for $h\in H$ and $(h_1,h_2)\in H\times H$, while $\overleftarrow{H}$ stands for its opposite. Let $A$ be a Green biset functor. The \textit{associated category of $A$}, denoted by $\mathcal{P}_A$, consists of the following data: 
\begin{itemize}
\item The objects of $\mathcal{P}_A$ are the same objects of $\mathcal{D}$.
\item If $K$ and $H$ are objects of $\mathcal{P}_A$, then $\mathcal{P}_A(H,K)=A(K\times H)$.
\item If $L$, $K$ and $H$ are objects of $\mathcal{P}_A$, then the composition of $\beta \in A(L\times K)$ and $\alpha \in A(K\times H)$ is defined as $\beta \circ \alpha=A(L\times \overleftarrow{K} \times H)(\beta \times \alpha)$.
\item For an object $H$ of $\mathcal{P}_A$, the identity morphism is $Id_H=A(\overrightarrow{H})(\epsilon_A) \in A(H\times H)$.
\end{itemize}
\end{deff}

It follows easily that the associated category is $k$-linear. The associated category was introduced by Bouc as a generalization of the biset category, and it turns out that $\mathcal{P}_{kB}$ and $k\mathcal{D}$ are isomorphic. We now get to the definition of module over a Green biset functor. For other definitions, see  Bouc \& Romero \cite[Subsection 1.2]{SerNad} or G. \cite[Section 2]{paper}.

\begin{deff}
A (left) \textit{$A$-module} is a $k$-linear functor from $\mathcal{P}_A$ to $k-Mod$. We write $A-\mathcal{M}od$ for the category of $A$-modules with morphisms given by natural transformations.
\end{deff}

The category $A-\mathcal{M}od$ is an abelian category, with direct sums, kernels and cokernels defined pointwise, and it generalizes $\mathcal{F}_{\mathcal{D},k}$, since the last is equivalent to $kB-\mathcal{M}od$. The functor $A$ is itself an $A$-module $_AA$ by the rule $_AA(\alpha)(a)=A(K\times \overleftarrow{H})(\alpha\times a)$ for $\alpha\in A(K\times H)$ and $a\in A(H)$. This structure can be restricted to left ideals and so these are exacly the $A$-submodules of $A$. 

Simple biset functors can be parametrized by means of equivalence pairs $(H,V)$ consisting of a group $H$ and a simple $kOut(H)$-module $V$. An attempt to generalize this technique for the classification of simple $A$-modules leads to the notions of \textit{minimal group}, \textit{essential algebra} and \textit{seeds}.

\begin{deff}
Let $M$ be an $A$-module. An object $H$ of $\mathcal{P}_A$ is a \textit{minimal group} for $M$ if $0\neq M(H)$ and $0=M(K)$ for any $K$ in $Ob(\mathcal{P}_A)$ such that $|K|<|H|$.
\end{deff}

Note that this definition agrees with the notion of minimal group for biset functors since $A$-modules are actually biset functors with an additional structure. 

\begin{deff}
Let $A$ be a Green biset functor and $H$ be an object of $\mathcal{P}_A$. The \textit{essential algebra of $A$ at $H$} is defined as the quotient
$$\widehat{A}(H)=\frac{End_{\mathcal{P}_A}(H)}{I_A(H)}$$
where $I_A(H)$ is the $k$-submodule of $A(H\times H)$ generated by the endomorphisms of $H$ which factor through groups of order strictly smaller than $|H|$. We write $\widehat{a}$ for the class of an element $a$ of $A(H\times H)$ in $\widehat{A}(H)$.
\end{deff}

The submodule $I_A(H)$ is a two-sided ideal of $End_{\mathcal{P}_A}(H)$, so $\widehat{A}(H)$ is in fact a $k$-algebra with identity $\widehat{Id_H}$. Nevertheless, $\widehat{A}(H)$ may still vanish for some groups $H$.

\begin{exx}
A \textit{seed} on $\mathcal{P}_A$ is a pair $(H,V)$ consisting of a group $H$ in $Ob(\mathcal{D})$ such that $0\neq \widehat{A}(H)$ and $V$ is a simple $\widehat{A}(H)$-module.
\end{exx}

If $(H,V)$ is a seed on $\mathcal{P}_A$, there is a way to construct a simple $A$-module $S_{H,V}$ having $H$ as a minimal group and $S_{H,V}(H)\cong V$ as a $\widehat{A}(H)$-modules. Furthermore, any simple $A$-module is isomorphic to one of this kind. There is an adequate notion of isomorphism of seeds, and if it holds that minimal groups for $A$-modules are unique up to group isomorphism, the assignation $(H,V)\mapsto S_{H,V}$ induces a bijection between the set of isomorphism classes of seeds and the set of isomorphism classes of simple modules. We refer to Romero \cite{ROM1} for details. 

\begin{rmk}
A parametrization by means of isomorphism classes of seeds is not always possible, as it was shown in Romero \cite[Section 2.1]{ROM2} for the fibered Burnside functor $kB^{C_4}$. Of course, seeds are not the only way to generate simple modules. Boltje and Coşkun provided in \cite{BolCos} a parametrization of the simple $kB^{A}$-modules by means of equivalence classes of cuadruples $(H,K,\kappa,[V])$ under the relation of \textit{linkage}.
\end{rmk}

By the previous paragraphs, a good first step for a classification of simple modules by classes of seeds is to determine whether $\widehat{A}(H)$ is zero or not for a given $H$. This motivates the following definition.

\begin{deff}
For a Green biset functor $A$, its \textit{essential support} is defined to be the subclass $\ESSUP(A)$ of $Ob(\mathcal{D})$ consisting of all the objects $H$ for which $0\neq \widehat{A}(H)$, and its \textit{essential nullity} $\ESZERO(A)$ to be the complement class of $\ESSUP(A)$.
\end{deff}

For the Burnside functor, $\widehat{kB}(H) \cong kOut(H)$ for any $H$, therefore $\ESSUP(kB)=Ob(\mathcal{D})$. We recall some other examples in the language of essential supports.

\begin{exx}
The functor of rational representations $\RQ$ and its modules, known as rhetorical biset functors, have been studied by Barker in \cite{LB}. He proved that when $k$ is a field of characteristic zero, $H$ lies in $\ESSUP(\RQ)$ if and only if $H$ is cyclic and there exists a primitive $k$Out$(H)$-module. Further computations show that $H$ lies in $\ESSUP(\RQ)$ if and only if $H$ is cyclic and $|H|\not\equiv 2\; mod\; 4$ (G. \cite[Lemma 3.17]{paper}).
\end{exx}

\begin{exx}\thlabel{clsfunctsupp}
For the functor $\CR$ of complex class functions, $\ESSUP(\CR)$ is just the isomorphism class of the trivial group, as it was proved by Romero in \cite[Proposition 4.3]{ROM1}.
\end{exx}

\begin{exx}
The functor $kpp_{\F}$ of rings of $p$-permutation modules for algebraically closed fields $k$ and $\F$ with $char\;k=0$ and $char\;\F=p>0$ has been studied by Ducellier in \cite{MD}. He proved in \cite[Proposition 4.1.2]{MD} that if $H\in \ESSUP(kpp_{\F})$, then $H= P\langle x\rangle$ for a normal $p$-subgroup $P$ of $H$ and a $p'$-element $x$ in $H$ such that $C_{\langle x\rangle }(P)=1$.
\end{exx}

\begin{exx}
For a general Green biset functor $0\neq A$, $\ESSUP(A)$ is never empty, since $0\neq A(1)\cong \widehat{A}(1)$ (G. \cite[Lemma 2.12]{paper}), and so $1\in \ESSUP(A)$. Moreover, the simple $A$-modules having trivial minimal group are in bijection with the simple $A(1)$-modules.
\end{exx}

Morphisms of Green biset functors may throw some light on the essential supports. If $f:A\longrightarrow C$ is a morphism of Green biset functors, the following properties can be easily verified:

\begin{enumerate}
\item $f_{L\times H}(\beta \circ \alpha) =f_{L\times K}(\beta)\circ f_{K\times H}( \alpha)$ for all morphisms $K \xrightarrow{\beta} L$ and $H \xrightarrow{\alpha} K$ in $\mathcal{P}_A$,
\item $f_{H\times H}(A(\overrightarrow{H})(\epsilon_A))=C(\overrightarrow{H})(\epsilon_C)$ for any $H \in Ob(\mathcal{P}_A)$.
\end{enumerate}
 
As a consequence, $f$ induces a $k$-linear functor between the associated categories
$$\mathcal{P}_f: \mathcal{P}_A \longrightarrow \mathcal{P}_C$$ 
by setting $\mathcal{P}_f(H)=H$ for any $H$ in $Ob(\mathcal{P}_A)$ and $\mathcal{P}_f(\alpha)=f_{K\times H}(\alpha)$ for any morphism $H \xrightarrow{\alpha} K$ in $\mathcal{P}_{A}$, which on its own induces a \textit{pullback functor}
$$f^*:C-\mathcal{M}od\longrightarrow A-\mathcal{M}od$$
given by $f^*M(H)=M(\mathcal{P}_f(H))=M(H)$ and $f^*M(\alpha)=M(\mathcal{P}_f(\alpha))=M(f(\alpha))$ for any $C$-module $M$. Also, for any $H$ we get a morphism of $k$-algebras between the endomorphism algebras
$$P_f:End_{\mathcal{P}_A}(H)\longrightarrow End_{\mathcal{P}_C}(H)$$
given by $\alpha \mapsto f_{H\times H}(\alpha)$, which by property (1) sends $I_A(H)$ into $I_C(H)$, thus inducing a homomorphism of $k$-algebras between the essential algebras 
$$\widehat{f}_H:\widehat{A}(H)\longrightarrow \widehat{C}(H)$$
by the rule $\widehat{\alpha}\mapsto \widehat{f_{H\times H}(\alpha)}$. This implies that if $0=\widehat{A}(H)$, then $0=\widehat{C}(H)$, or equivalently, if $0\neq \widehat{C}(H)$, then $0\neq \widehat{A}(H)$. In terms of essential supports, the previous discussion proves the following result.

\begin{lemma} \thlabel{ESSUPmorph}
If $f:A\longrightarrow C$ is a morphism of Green biset functors, then $\ESSUP(C)\subset \ESSUP(A)$ and $\ESZERO(A)\subset \ESZERO(C)$.
\end{lemma}

This result, although easy, turns out to be very helpful when looking for hints on the essential supports. It has already been used implicitly in the proof of G. \cite[Proposition 3.2.1]{paper}, which states that $\ESSUP(\RF)\subset \ESSUP(\RQ)$, by considering the extension morphism $^{\F}\eta:\RQ\longrightarrow \RF$.

\section{Shifted Green biset functors}

If $L$, $K$, $H$ and $G$ are objects in $\mathcal{D}$, and $_LY_K$ and $_HX_G$ are bisets, then the \textit{external product of $Y$ and $X$} is the product $Y\times X$ as a $(L\times H,K\times G)$-biset by the rule $(l,h)\cdot(y,x)\cdot(k,g)=(l\cdot y\cdot h,h\cdot x\cdot g)$ for $(y,x)\in Y\times X$, $(l,h)\in L\times H$ and $(k,g)\in K\times G$. This construction extends to a bilinear map $kB(L,K)\times kB(H,G)\longrightarrow kB(L\times H,K\times G)$, denoted by $(\alpha,\beta)\mapsto \alpha\times \beta$.

\begin{deff}
Let $G$ be an object of $k\mathcal{D}$ and $F$ be a biset functor. The \textit{Yoneda-Dress construction of $F$ at $G$}, or \textit{$F$ shifted by $G$} for short, is defined to be the biset functor $F_G$ whose value at an object $H$ of $k\mathcal{D}$ is $F_G(H)=F(H\times G)$, and $F_G(\alpha)=F(\alpha\times Id_G): F_G(H)\longrightarrow F_G(K)$ for any morphism $\alpha:H\longrightarrow K$ in $k\mathcal{D}$.
\end{deff}

The Yoneda-Dress construction at $G$ defines then an endofunctor $P_G: \mathcal{F}_{\mathcal{D},k}\longrightarrow \mathcal{F}_{\mathcal{D},k}$, given by $P_G(F)=F_G$ for any biset functor $F$, and $P_G(f):F_G\longrightarrow F'_G$ for any morphism of biset functors $f:F\longrightarrow F'$, whose $H$-th component arrow is $P_G(f)_H=f_{H\times G}$ for any $H$ in $k\mathcal{D}$. 

For a general Green biset functor $A$ and any object $G$ of $\mathcal{D}$, the shifted functor $A_G$ is a Green biset functor for the products $\times^d:A_G(K)\times A_G(H) \longrightarrow A_G(K\times H)$ given by 
$$a\times^d b=A(_{K\times H\times G^{\Delta}}K\times G\times H\times G_{K\times G\times H\times G})(a\times b)$$
for all $a\in A(K\times G)$ and $b\in A(H\times G)$ and all $K,H\in Ob(\mathcal{D})$, and with identity element $\epsilon_{A_G}=A(Inf^{1\times G}_1)(\epsilon_A)$. 

Recently, there have been an increasing interest on the shifted Green biset functors $A_G$, as these are both Green biset functors and projective $A$-modules, and as $G$ runs over a set of representatives of the isomorphim classes of objects of $\mathcal{D}$, we get a set of projective generators of $A-\mathcal{M}od$ (see G. \cite[Subsection 2.3]{paper}). At the same time, little is known on their essential algebras. Nevertheless, we can prove that $\ESSUP(A_G) = \ESSUP(A)$ for any $G$. To do this, let introduce a morphism of Green biset functors from $A$ to $A_G$.

\begin{prop} \thlabel{INF}
Let $A$ be a Green biset functor and $G$ be an object of $\mathcal{D}$. Then the morphisms 
$$Inf_H=A(Inf_H^{H\times G}):A(H)\longrightarrow A(H\times G)$$
for $H$ in $Ob(\mathcal{P}_A)$ define a morphism of Green biset functors $Inf: A\longrightarrow A_G$.
\end{prop}
\begin{proof}
We prove first that $Inf$ is a morphism of biset functors. By Bouc \cite[8.2.4 \& 8.2.7]{SB1}, 
$$Id_{\mathcal{F}_{\mathcal{D},k}}\cong P_1 \xrightarrow{P_{Inf^G_1}} P_G: \mathcal{F}_{\mathcal{D},k} \longrightarrow \mathcal{F}_{\mathcal{D},k}$$
is a natural transformation, whose $A$-th component arrow is the morphism of biset functors
$$A\xrightarrow{A(\_\times Inf^G_1)\circ A(Iso(\rho_{\_}^{-1}))} A_G,$$
with $H$-th component arrow $A(H\times Inf^G_1)\circ A(Iso(\rho_H^{-1}))=A(Inf^{H\times G}_H)=Inf_H$, therefore $Inf=A(\_\times Inf^G_1)\circ A(Iso(\rho_{\_}^{-1}))$.

Now we prove that $Inf$ is actually a morphism of Green biset functors. If $a\in A(K)$ and $b\in A(H)$, then
$$Inf_{K\times H}(a\times b)=A(Inf^{K\times H\times G}_{K\times H})(a\times b)$$
and
$$Inf_K(a)\times^dInf_H(b)=A(_{K\times H\times G^{\Delta}}K\times G\times H\times G_{K\times G\times H\times G}\circ Inf^{K\times G\times H\times G}_{K\times H})(a\times b),$$
and it is straightforward to verify that the map
$$_{K\times H\times G^{\Delta}}K\times G\times H\times G_{K\times G\times H\times G}\circ Inf^{K\times G\times H\times G}_{K\times H} \longrightarrow Inf^{K\times H\times G}_{K\times H}$$
$$[(k_1,g_1,h_1,g_2),(k_2,h_2)]\mapsto (k_1k_2,h_1h_2)$$
is an isomorphism of $(K\times H\times G,K\times H)$-bisets, so it follows that $Inf_{K\times H}(a\times b)=Inf_K(a)\times^dInf_H(b)$, and we already had that $\epsilon_{A_G}=A(Inf^{1\times G}_1)(\epsilon_A)=Inf_1(\epsilon_A)$.
\end{proof}

We call $Inf$ the \textit{inflation morphism from $A$ to $A_G$}. Now we give a morphism in the other direction. For simplicity, we will write $Res^{H\times G}_H$ instead of $Iso(\rho_H)\circ Res^{H\times G}_{H\times 1}$ for any $H\in Ob(\mathcal{D})$.

\begin{prop} \thlabel{RES}
Let $A$ be a Green biset functor and $G$ be an object of $\mathcal{D}$. Then the morphisms
$$Res_H=A(Res_H^{H\times G}):A(H\times G)\longrightarrow A(H)$$
for $H$ in $Ob(\mathcal{P}_A)$ define a morphism of Green biset functors $Res: A_G \longrightarrow A$.
\end{prop}
\begin{proof}
First we prove that $Res$ is a morphism of biset functors. By Bouc \cite[8.2.4 \& 8.2.7]{SB1}, 
$$P_G \xrightarrow{P_{Res^G_1}} P_1\cong Id_{\mathcal{F}_{\mathcal{D},k}}: \mathcal{F}_{\mathcal{D},k} \longrightarrow \mathcal{F}_{\mathcal{D},k}$$
is a natural transformation whose $A$-th component arrow is the morphism of biset functors
$$A_G\xrightarrow{A(Iso(\rho_{\_}))\circ A(\_\times Res^G_1)} A,$$
with $H$-th component given by $A(Iso(\rho_{H}))\circ A(H\times Res^G_1)=A(Res^{H\times G}_H)=Res_H$, and so $Res=A(Iso(\rho_{\_}))\circ A(\_\times Res^G_1)$.

Now we prove that $Res$ is a morphism of Green biset functors. If $a\in A(K\times G)$ and $b\in A(H\times G)$, then
$$Res_{K\times H}(a\times^d b)=A(Res^{K\times H\times G}_{K\times H}\circ _{K\times H\times G^{\Delta}}K\times G\times H\times G_{K\times G\times H\times G})(a\times b),$$
$$Res_K(a)\times Res_H(b)=A(Res^{K\times G}_{K}\times Res^{H\times G}_{H})(a\times b)=A(Res^{K\times G\times H\times G}_{K\times H})(a\times b),$$
and it is straightforward that
$$Res^{K\times H\times G}_{K\times H}\circ _{K\times H\times G^{\Delta}}K\times G\times H\times G_{K\times G\times H\times G} \longrightarrow Res^{K\times G\times H\times G}_{K\times H}$$
$$[(h_1,k_1,g_1),(h_2,g_2,k_2,g_3)]\mapsto (h_1h_2,g_1g_2,k_1k_2,g_1g_3)$$
is an isomorphism of $(K\times H,K\times G\times H\times G)$-bisets, therefore $Res_{K\times H}(a\times^d b)=Res_K(a)\times Res_H(b)$. Finally, $Res_1(\epsilon_{A_G})= A(Res^{1\times G}_1\circ Inf^{1\times G}_1)(\epsilon_A)=\epsilon_A$ since $Res^{1\times G}_1\circ Inf^{1\times G}_1\cong 1$.
\end{proof}

We call $Res$ the \textit{restriction morphism from $A_G$ to $A$}. The morphisms $Inf$ and $Res$ together with \thref{ESSUPmorph} imply the following result.

\begin{cor} \thlabel{shiftedSUP}
Let $A$ be a Green biset functor. Then $\ESSUP(A_G)=\ESSUP(A)$ for all $G$ in $Ob(\mathcal{D})$.
\end{cor}

This corollary together with $\ESSUP(kB)=Ob(\mathcal{D})$ give a different and shorter proof of Romero \cite[Lemma 4.10]{ROM1}, which asserts that $0\neq\widehat{kB_G}(H)$ for all $H$ and $G$ in $Ob(\mathcal{D})$. We have further implications of these morphisms.

\begin{prop} \thlabel{splitAG}
Let $A$ be a Green biset functor and $G$ be in $Ob(\mathcal{D})$. There is an isomorphism of $A$-modules
$$Inf^*A_G\cong A\oplus \kappa_G,$$
where $\kappa_G$ is the kernel of $Res$ which is a bilateral ideal of $A_G$. Furthermore, for any $H$ we get an isomorphism 
$$End_{\mathcal{P}_{A_G}}(H)\cong End_{\mathcal{P}_{A}}(H)\oplus \kappa_G(H\times H)$$ 
of $End_{\mathcal{P}_A}(H)$-modules, and passing to the quotient, we get an isomorphism
$$\widehat{A_G}(H)\cong \widehat{A}(H)\oplus \widehat{\kappa_G}(H)$$
of $\widehat{A}(H)$-modules, where $\widehat{\kappa_G}(H)$ is the image of $\kappa_G(H\times H)$ in $\widehat{A_G}(H)$.
\end{prop}
\begin{proof}
Since $Res_H^{H\times G}\circ Inf_H^{H\times G} \cong Id_H$ for any $H$, it follows that $Res\circ Inf=Id_A$, and in particular, $Res_H$ is surjective and $Inf_H$ is injective. Then we have an exact sequence
$$0\longrightarrow \kappa_G \longrightarrow A_G\xrightarrow{Res} A \longrightarrow 0$$
of $A_G$-modules. Since $Inf:A\longrightarrow Inf^* A_G$ is a morphism of $A$-modules, we have that
$$0\longrightarrow \kappa_G \longrightarrow Inf^*A_G\xrightarrow{Res} A \longrightarrow 0$$
splits by $Inf$ in $A-\mathcal{M}od$. Thus the first part follows. The other isomorphisms follow from the fact that 
$$Res_{H\times H}:End_{\mathcal{P}_{A_G}}(H)\longrightarrow End_{\mathcal{P}_{A}}(H)$$ 
is a surjective $k$-algebra homomorphism for the algebra estructure given by composition, and it has $Inf_{H\times H}$ as a section.
\end{proof}

\begin{rmk}
We must be aware that the decomposition $Inf^*A_G\cong A\oplus \kappa_G$ is not a decomposition of $A_G$-modules since $Inf(A)$ is not an ideal of $A_G$. However, the ideal $\kappa_G$ is still interesting since it splits the set of isomorphism classes of simple $A_G$-modules into two: the simple $A_G$-modules which are obtained by pulling back simple $A$-modules, and those which are not annihilated by $\kappa_G$.
\end{rmk}

\begin{prop}
Let $A$ be a Green biset functor and $G$ be an object of $\mathcal{D}$. 
\begin{enumerate}
    \item The pullback functor $Res^*:A-\mathcal{M}od \longrightarrow A_G-\mathcal{M}od$ is faithful and $Res^*M\cong Res^*N$ if and only if $M\cong N$ given $M$ and $N$ in $A-\mathcal{M}od$.
    \item The functor $Res^*$ induces a bijection between the set of isomorphism classes of simple $A$-modules and the set of isomorphism classes of simple $A_G$-modules annihilated by $\kappa_G$, sending the class of a simple $A$-module $S$ to the class of $Res^*S$. Moreover, if $H$ is a minimal group for $S$, then $S(H)$ is naturally a simple $\widehat{A}(H)$-module.
    \item If $S$ is a simple $A_G$-module that is not annihilated by $\kappa_G$ and $H$ is a minimal group for $S$, then $S(H)$ is a simple quotient of $\widehat{\kappa_G}(H)$ as $\widehat{A_G}(H)$-modules.
\end{enumerate}
\end{prop}

\begin{proof}
\begin{enumerate}
    \item It follows easily since $Inf^*\circ Res^*=(Res\circ Inf)^*= Id_{A-\mathcal{M}od}$.
    \item The bijection is a consequence of part 1. Now, if $S$ is annihilated by $\kappa_G$, then $S(H)$ as a simple $\widehat{A_G}(H)$-module is annihilated by $\widehat{\kappa_G}(H)$ and thus it is a simple $\widehat{A}(H)$-module.
    \item Since $S$ is not annihilated by $\kappa_G$, then $S(H)=S(\kappa_G(H\times H))(H)=\widehat{\kappa_G}(H)S(H)$. If $S(H)=\langle s\rangle$ as an $\widehat{A_G}(H)$-module, then it follows that $\widehat{\kappa_G}\longrightarrow S(H)$ given by $x\mapsto xv$ is a surjective homomorphism of $\widehat{A_G}(H)$-modules.
\end{enumerate}
\end{proof}

\begin{exx}
Let $A=\RFG$ for fields $k$ and $\F$ of characteristic zero. We know from G. \cite[Proposition 3.21]{paper} that $\ESSUP(\RFG)$ has only cyclic groups, so the isomorphisms classes of its simple modules are in bijection with the isomorphism classes of seeds. Since $\kappa_G$ is an ideal of $\RFG$, from G. \cite[Theorem 3.12]{paper} we have that $\kappa_G=I_{\mathcal{E}}$ for some $\mathcal{E}\subset c_{k\cap \F}(G)$, where $c_{k\cap \F}(G)$ is the set of $k\cap \F$-conjugacy classes of $G$ and $I_{\mathcal{E}}(H)$ is the ideal of $\RF(H\times G)$ generated by the primitive idempotents $e_D$ for $D\in c_{k\cap \F}(H\times G)$ such that the projection of $D$ to $G$ is an element of $\mathcal{E}$. Thus, $\kappa_G$ depends only on $\kappa_G(1)$. If $D\in c_{k\cap \F}(1\times G)$, then by Bouc \cite[Lemma 7.1.3, part 2]{SB1}, we have
\begin{align*}
    Res_1(e_D)(1) &=\RF(_11\times G_{1\times G})(e_D)(1)=\frac{1}{|1\times G|}\sum_{\substack{(1,g_1)\in _11\times G_{1\times G}, (1,g_2) \in 1\times G\\
(1,g_1)=(1,g_1g_2)}} e_D(1,g_2)\\
&=\frac{1}{|G|}\sum_{(1,g_1)\in _11\times G_{1\times G}} e_D(1,1)= 
     \begin{cases}
       \text{1,} &\quad\text{if $D=\{(1,1)\} $}\\
       \text{0,} &\quad\text{otherwise} \\ 
     \end{cases}
\end{align*}
proving that $\mathcal{E}=c_{k\cap\F}(G)-\{\{1\}\}$ and so $\kappa_G=I_{c_{k\cap\F}(G)-\{\{1\}\}}$. In this case, we have that $A=Res^*A$ and $I_{\{1\}}$ are isomorphic as $A_G$-modules, and therefore
$$\RFG=A\oplus \kappa_G$$
as $A_G$-modules, even though $I_{\{1\}}$ and $Inf(A)$ do not agree as subfunctors of $A_G$. When $\F=\Q$, the $\RQG$-modules are known as \textit{$G$-rhetorical biset functors}. By G. \cite[Proposition 3.23]{paper}, if $(|H|,|G|)=1$ then the map $\nu:\ER(H)\otimes_k \RQ(G)\longrightarrow \ERG(H)$ given by $\nu(\widehat{a}\otimes b)=\widehat{a\times b}$ is a $k$-algebra isomorphism. Since $\nu(\widehat{e_D}\otimes e_E)=\widehat{e_{D\times E}}$ for $D\in c_{\Q}(H\times H)$ and $E\in c_{\Q}(G)$, we obtain
$$\widehat{\kappa_G}(H)=\nu\left(\ER(H)\otimes_k \left\langle \left\{ e_E\;|\;\{1\}\neq E\in c_{\Q}(G)\right\}\right\rangle\right)=\left\langle \left\{\widehat{Id_H\times e_E}\;|\;\{1\}\neq E\in c_{\Q}(G)\right\}\right\rangle.$$
Under the bijection given in G. \cite[Corollary 3.24]{paper}, the isomorphism classes of simple $G$-rhetorical biset functors obtained from simple rhetorical biset functors via $Res^*$ correspond to classes of triplets of the form $(H,V,\{1\})$, while those which are not annihilated by $\kappa_G$ correspond to classes of those $(H,V,C)$ with $C$ a nontrivial cyclic subgroup of $G$.
\end{exx}

On the support of $\RFG$, G. \cite[Proposition 3.21]{paper} asserts that $\ESSUP(\RFG)\subset\ESSUP(\RQ)$. Here there is a slight improvement on this result.

\begin{prop} \thlabel{SUPRFG}
Let $G$ be an object of $\mathcal{D}$ and $\F$ be a field of characteristic zero. Then 
$$\ESSUP(\RF)=\ESSUP(\RFG)\subset\ESSUP(\RQG)=\ESSUP(\RQ).$$
Furthermore, if $H\in Ob(\mathcal{D})$ is a non-trivial cyclic group and $\F$ contains an $|H|$-th primitive root of $1$, then $H\in \ESZERO(\RF)$.
\end{prop}

\begin{proof}
The contention in the middle was already know to G. \cite[Proposition 3.21]{paper}, while the equalities follow immediately from \thref{shiftedSUP}. Now, if $\F$ contains a $|H|$-th primitive root of 1, then it is a splitting field for $H$, and \cite[Theorem 10.33]{CR1} implies then that each simple $\F (H\times H)$-module $V$ is isomorphic to a external product $V_1\otimes_{\F}V_2$ for simple $\F H$-modules $V_1$ and $V_2$, and this last is isomorphic to $Inf^{H\times 1}_HV_1\circ Inf^{1\times H}_HV_2$. Since the simple modules generate $\RF(H\times H)$, we obtain $\widehat{\RF}(H)=0$.
\end{proof}

\begin{rmk}
If $\F$ contains roots of unit of any order, then $\ESSUP(\RF)$ is just the isomorphism class of the trivial group, and we get \thref{clsfunctsupp} as a particular case. Since proof of G. \cite[Proposition 4.3]{paper} does not make use of the fact that $k=\C$, it can be easily adapted to prove that for general fields of characteristic zero $k$ and $\F$ with $\F$ algebraically closed, the category $\RFG-\mathcal{M}od$ is equivalent to $\RF(G)-Mod$, where the last is a semisimple category since $\RF(G)$ is isomorphic to a product of finite field extensions of $k$.
\end{rmk}

\section*{Acknowledgement}

Part of the content in this note was included in my PhD thesis at the Posgrado Conjunto en Ciencias Matemáticas UNAM-UMSNH, written with the generous guidance of my advisors Gerardo Raggi-Cárdenas and Nadia Romero and supported by CONACyT, grant 384063; a significant amount of this work was later developed during my current Postdoctoral Fellowship at the Department of Mathematics of the University of California Santa Cruz, sponsored by UCMEXUS. I want to thank to Robert Boltje for his remarks on the relation between the categories of $A-\mathcal{M}od$ and $A_G-\mathcal{M}od$ induced by pullback via $Res$.

\end{document}